\documentclass[12pt]{amsart}
\usepackage[utf8]{inputenc}
\usepackage{hyperref,nameref,cleveref}
\usepackage{xcolor}
\usepackage{enumitem,kantlipsum} 
\usepackage{mathrsfs} 
\usepackage{stmaryrd} 
\usepackage{mathtools, amssymb} 
\counterwithin{equation}{section} 
\usepackage{epsfig} 
\usepackage[justification=centering]{caption} 
\usepackage{setspace} 
\usepackage[margin=1.2in]{geometry}

\usepackage{ragged2e}
\usepackage[
backend=biber,
sorting=nyt
]{biblatex}
\newtheorem{Theorem}{Theorem}[section]

\newtheorem{Lemma}[Theorem]{Lemma}
\newtheorem{Conjecture}[Theorem]{Conjecture}

\theoremstyle{definition}
\newtheorem{Definition}[Theorem]{Definition}
\newtheorem{Example}[Theorem]{Example}

\theoremstyle{remark}
\newtheorem{Remark}[Theorem]{Remark}

\newcommand{\cB}{\mathcal{B}}
\newcommand{\cC}{\mathcal{C}}

\newcommand{\cF}{\mathcal{F}}

\newcommand{\cH}{\mathcal{H}}
\newcommand{\cI}{\mathcal{I}}

\newcommand{\cM}{\mathcal{M}}
\newcommand{\cN}{\mathcal{N}}

\newcommand{\A}{\mathbb{A}}
\newcommand{\R}{\mathbb{R}}
\newcommand{\B}{\mathbb{B}}

\newcommand{\Z}{\mathbb{Z}}
\newcommand{\M}{M}
\let\S=\undefined\newcommand{\S}{\mathbb{S}}
\newcommand{\N}{\mathbb{N}}

\newcommand{\F}{\mathcal{F}}

\newcommand{\e}{\epsilon}

\newcommand{\wtilde}[1]{\widetilde{#1} } 
\newcommand{\Prod}[2]{\ensuremath{\langle #1,#2 \rangle}}
\newcommand\res[2]{{
  \left.\kern-\nulldelimiterspace 
  #1 
  \vphantom{\big|} 
  \right|_{#2} 
  }}

\newcommand{\mres}{\mathbin{\vrule height 1.6ex depth 0pt width
0.13ex\vrule height 0.13ex depth 0pt width 1.3ex}}
\let\P=\undefined\newcommand{\P}{\mathbf{P}}
\let\M=\undefined\newcommand{\M}{\mathbf{M}}
\let\a=\undefined \newcommand{\a}{\alpha}

\let\d=\undefined \newcommand{\d}{\delta}
\let\p=\undefined \newcommand{\p}{\partial}
 
\let\o=\undefined \newcommand{\o}{\omega}
\let\O=\undefined \newcommand{\O}{\Omega}
\let\H=\undefined \newcommand{\H}{\mathcal{H}}

\let\dis\displaystyle

\makeatletter
\newcommand{\leqnomode}{\tagsleft@true}
\newcommand{\reqnomode}{\tagsleft@false}
\makeatother
\reqnomode

\DeclarePairedDelimiter\abs{\lvert}{\rvert}%
\DeclarePairedDelimiter\norm{\lVert}{\rVert}%

\makeatletter
\let\oldabs\abs
\def\abs{\@ifstar{\oldabs}{\oldabs*}}

\let\oldnorm\norm
\def\norm{\@ifstar{\oldnorm}{\oldnorm*}}
\makeatother

\let\div=\undefined \DeclareMathOperator{\div}{div}

\DeclareMathOperator{\diam}{diam}
\DeclareMathOperator{\Vol}{Vol}

\DeclareMathOperator{\spt}{spt}

\title[Singular isoperimetric regions]{Existence of singular isoperimetric regions}

\author{GONGPING NIU}
\address{GN: Department of Mathematics, University of California, San Diego, 9500 Gilman Dr, La Jolla, CA 92093, USA}
\email{gniu@ucsd.edu}

\begin{document}
    \begin{abstract}
        It is well known that isoperimetric regions in a smooth compact $(n+1)$-manifold are smooth, up to a closed set of codimension at most $6$. In this note, we first construct an $8$-dimensional compact smooth manifold whose unique isoperimetric region with half volume that of the manifold exhibits two isolated singularities. And then, for $n\geq 7$, using Smale's construction of singular homological area minimizers for higher dimensions, we construct a Riemannian manifold such that the unique isoperimetric region of half volume, with singular set the submanifold $\S^{n-7}$.
\end{abstract}

\maketitle


\section{Introduction}
    Given an $(n+1)$-dimensional smooth compact Riemannian manifold $(M,g)$, we will consider the isoperimetric problem, that is, 
    given any positive number $0<t<|M|_g$, we look for a solution to the following constrained variational problem
    \begin{align}
        \cI(M, g, t):=\inf\{\P_g(\Omega)\, :\,\Omega \in \cC(M,t)\} \label{minimizer} \,,
    \end{align}
    where $\cC(M,t)$ is the class of sets of perimeters with enclosed volume $t$ (for a more precise definition see section~\ref{Section: Preliminaries}.) To understand the structure of the boundary of $\O$, we say that a point on the boundary is regular if it is locally a smooth hypersurface. We denote it as $Reg(\p \O)\subset \p\O$, and for the complement, we call it the set of singularities in $\p\O$, and denote it by $Sing(\p\O)=\p\O \setminus Reg(\p \O)$. It is well known that minimizers exist and that they are regular outside of a closed set of codimension $6$, which is discrete when $(n+1)=8$ (see Gonzalez, Massari, and Tamanini \cite{gonzalez1983regularity}). 
    
    A natural question is whether the bound is optimal or not. In the case of area-minimizing integral currents, the question is answered in the positive by the Simons' cone in $\R^8$: ${\bf{C}}:=\{ (x,y) : |x|=|y| \ \ \text{for} \ \ x,y\in \R^4  \}.$
    However, minimizers of the problem above in $\R^n$ are euclidean balls and hence smooth for every $n\in \N$. Therefore, to construct a singular minimizer in dimension 8, we need to construct a manifold that is not a space form (see more explanation in section~\ref{Section: Preliminaries}.) On the other hand, Smale constructed in \cite{smale1999singular} a compact $8$-manifold admitting a unique area-minimizing current with two singular points. Following his construction of local minimizing neighborhood (see section~\ref{subection: Smale's result}), we can prove the following theorem for the isoperimetric region problem:

    \begin{Theorem}[Singular isoperimetric region in 8-manifold]\label{Theorem: main theorem}
    There exists a smooth closed Riemannian $8$-manifold $(M,g)$ whose unique isoperimetric region with volume $|M|_g/2$ has two isolated singularities. The unique tangent cone at each singular point is a Simons' cone.
    \end{Theorem}
    
 \newpage 
    
    \begin{Remark}\label{Remark: for the main theorem}
    \hfill
    
    \begin{itemize}
        \item In \cite{smale1999singular}, we can prescribe the singularity for the homological area minimizers to be any strictly stable, strictly minimizing (tangent) cone with an isolated singularity (see section~\ref{Section: Preliminaries} for the definitions), but in our construction, for the technical reason of the setting, we need in addition to assume that the unique (up to scaling) smooth area minimizing hypersurface on one side of the cone (see \cite[Theorem 2.1]{hardt1985area}) is diffeomorphic to the one on the other side (e.g., Simons' cone). However, it is promising that we can eliminate this requirement by modifying the construction. 
        As far as the author's knowledge, these are the first examples of isoperimetric regions with singularities.     
        \item The metric in the above theorem is only $C^\infty$. It is an open question whether the same result would hold for an analytic metric.
    \end{itemize}
    \end{Remark}
   
   Another natural problem is whether we have examples of singular isoperimetric regions for higher dimensions. By slightly modifying the proof in Theorem \ref{Theorem: main theorem}, we can generalize it to higher dimensions.

    \begin{Theorem}[Singular isoperimetric regions in higher dimensional manifolds] \label{Theorem: main theorem 2}
        For any integers $n\geq 7,  p\in[3, \frac{n-1}{2}]$, there exists a closed smooth $(n+1)$-dimensional Riemannian manifold $(M,g)$ such that whose unique isoperimetric region $\O$ with volume half has the singular part a closed submanifold diffeomorphic to $\S^{n-2p-1}$ (denote $\S^0$ as a point). Denote the Simons' cone in $\R^{2p+2}$, 
        $$ {\bf C}^{p,p}:= \{ (x,y)\in \R^{p+1}\times\R^{p+1}: \ |x| = |y| \}.$$
        
        Near $Sing(\p \O)$, $\p \O$ looks like $\S^{n-2p-1} \times {\bf C}^{p,p}$, i.e., there exists $\sigma>0$, and an isometric map $\Phi$ from the tubular neighborhood of the singular set $\cN(\sigma):= \{ x\in M : d_M(x,Sing(\Sigma)) <\sigma \}$ to $(\B^{2p+2}(\sigma) \times \S^{n-2p-1}, g_{eucl} + g_S)$, here $g_{eucl}, g_S$ are the standard metrics in $\R^{2p+2}$ and $\S^{n-2p-1}$ respectively. Moreover,
        $$  \Phi(\p\O\cap \cN(\sigma)) = {\bf C}^{p,p}(\sigma) \times \S^{n-2p-1}.  $$
    \end{Theorem}
    
    We could replace $\S^{n-2p-1}$ to quite numerous varieties as the singular part, with only some topological restrictions (see details in section \ref{Section: Proof of theorem 2}). The construction comes from the fruitful examples of singular homological area minimizing codimension 1 currents (see also \cite[Theorem A]{smale2000construction}). The proof of Theorem \ref{Theorem: main theorem 2} strongly uses the result from \cite{smale2000construction} and a slight modification of the proof in Theorem \ref{Theorem: main theorem}. In section \ref{section: Constructing the manifolds} to \ref{Section: Proof of the main theorem}, we will focus on constructing examples with isolated singularities. In section \ref{Section: Proof of theorem 2}, we will prove Theorem \ref{Theorem: main theorem 2}.


\subsection*{Idea of the construction for isolated singularities}
    \hfill
    
    As in \cite{smale1999singular}, the starting point is constructing a singular minimal surface in $\mathbb S^8$. We denote ${\bf{C}}$ as a 7-dimensional Simons' cone. Then the product space $\textbf{C} \times \R$ will be an area-minimizing cone in $\R^9$. Now consider $\Sigma={\bf{C}}\times \R \cap \mathbb S^8$ in $\mathbb S^8$. Clearly, $\Sigma$ is a minimal hypersurface in $(\mathbb S^8, g_S)$ with two isolated singularities, where $g_{S}$ is the round metric.
    
    The important part of Smale's work is to prove Theorem~\ref{Lemma: Local homological minimizing} [cf. \cite[Lemma 4]{smale1999singular}]: under a conformal change of the standard metric of $\mathbb S^8$, there exists a smooth neighborhood $V$ of $\Sigma$ such that $\Sigma$ is the unique homological area-minimizing current in $V$. Moreover, $\Sigma$ splits $\overline{V}$ into two parts, $V_+$ and $V_-$, and $\p V$ has exactly two components (denoted by $\Gamma_+,\Gamma_-$) such that they lie in $\overline{V}_+, \overline{V}_-$ respectively. Each part of $\{\Gamma_\pm\}$ is homologous to $\Sigma$ and $\Gamma_+$ is diffeomorphic to $\Gamma_-$ (Theorem \ref{Theorem: V} below).

    Next, we will construct a manifold with a singular isoperimetric region. Consider $(\Gamma,g_l)$ the 7-manifold which is diffeomorphic to $\Gamma_+$ and $\Gamma_-$, endowed with a ``larger'' metric (see Theorem \ref{Theorem: Get M(R)} for details); denote $T_R:=\Gamma\times [0,R]$ a tube with length $R>0$, with the product metric $g=g_l + dr^2$. We glue the $\Gamma_+,\Gamma_-$ with the boundary of the tube, $\wtilde{\Gamma}_-:=\Gamma\times\{0\},\wtilde{\Gamma}_+:=\Gamma\times\{R\}$ respectively to form a torus, calling it $M_R$. Finally, we prove that (in section~\ref{Section: Proof of the main theorem}) for sufficiently large $R$, the unique isoperimetric region with half the volume of $M_R$ has boundary $(\Gamma\times \{t_0\}) \cup \Sigma$ for some $t_0\in (0,R)$, where $\Sigma$ is the one described in the previous paragraph, with two isolated singularities. 
    
    \begin{figure} [ht] 
        \centering
        \includegraphics{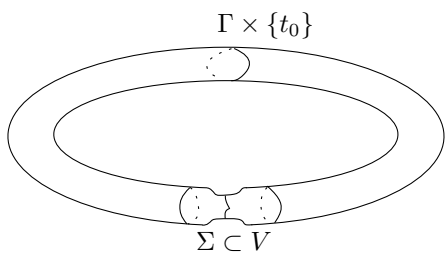}
        \caption{$M(R):= \Gamma\times [0,R]\cup V/\Gamma_{\pm} \sim \tilde \Gamma_{\pm}.$ }
        \label{Figure: Torus}
    \end{figure}
    \subsection*{Acknowledgement}
    I am thankful to my two advisors, Professors Bennett Chow and Luca Spolaor, for their instruction in geometric analysis. I am especially indebted to Professor Spolaor for his tremendous advice and suggestions on this problem and on geometric measure theory in general, which made this paper possible. Also, I would like to thank Zhenhua Liu for his valuable comments and corrections and Davide Parise for his helpful suggestions.

\section{Preliminaries}\label{Section: Preliminaries}

\subsection*{Sets of finite perimeters}
\hfill

    Let $n\geq 1$ and $(M, g)$ be an $(n+1)$-dimensional closed oriented Riemannian manifold. In most cases, we will consider $n=7$. In order to have an explicit definition of our variational problem, we first review the sets with well-defined measure theoretical perimeters.

    \begin{Definition}[Caccioppoli sets/sets of finite perimeter; see e.g.  \cite{munoz2020isoperimetric}]
        Suppose $E$ is a Lebesgue measurable subset of $M$, we define the perimeter of $E$ by
        $$\P_g(E)=\sup\left\{ \int_E \div_g X \ d\cH^{n+1} (x)\ : \ X\in \Gamma^1 (M) , \ \norm{X}_{g}\le 1 \right\}.$$       
        We define the collection of sets of finite perimeters in $(M,g)$ by 
        $$\cC(M,g):= \{\O\subset M\,:\, \P_g(\O)<\infty \}.$$
        In addition, for the subset of $\cC(M)$ with a fixed volume, we denote
        $$\cC(M,g,t):= \{\O\subset M\,:\, \P_g(\O)<\infty\,,\,\, |\O|_g=t\}.$$
        We usually omit the subscript $g$ above if the 
        defining 
        metric is understood.
    \end{Definition}

    By the Riesz Representation Theorem (see e.g. 
     \cite{simon1983lectures}), there is a $TM$-valued Radon measure $\mu_\O$ such that for any $X\in \Gamma^1 (M) $, we have
    $$ \int_\O \div_g X= \int_{M} X\cdot_g  d\mu_\O.$$

    The total variation is denoted by $\norm{\mu_\O}$. 
    For an open set $U$, we denote 
    $$\P(\O;U)=\norm{\mu_\O}(U),$$ 
    the relative perimeter in $U$.
    
    By de Giorgi's structure theorem, there is a $n$-rectifiable set $\p^* \O$ and $\nu_\O$ a Borel function on $\p^* \O$ with unit sphere such that 
        $$\mu_\O=\nu_\O \cH^n\mres\p^* \O. $$
    
    For simplicity, we will denote $\p \O = \spt \mu_\O$.


\subsection*{Isoperimetric problems and minimal surfaces}
\hfill

    In this paper, we will study the following minimizing problem:
    \begin{align*}
        \cI(M, g, t):=\inf\{\P_g(\Omega)\, :\,\Omega\in \cC(M,g,t) \}, 
    \end{align*}
    where $(M^8,g)$ is a closed Riemannian manifold and $t\in (0, |\O|_g)$.
    We will call a Cacciopoppli set $\O$\ a \textbf{$t$-isoperimetric region} in $(M,g)$ if 
    $$\P_g(\O)= \cI(M,g,t)\ \ \text{ and } \ \ \O \in \cC(M,g,t).  $$
    Intuitively, the boundary of $\O$ is a $n$-dimensional hypersurface. 
    \begin{Definition}
        Let $\Sigma\subset M$ be a smooth open hypersurface. We denote 
        \begin{align*}
            Reg(\Sigma)\ &:= \{x\in \overline{\Sigma}\cap U: \overline{\Sigma} \text{ is smooth, embedded hypersurface near }x \}; \\
            Sing(\Sigma)& := \overline{\Sigma}\setminus Reg(\Sigma). 
        \end{align*}
        By adding points to $\Sigma$ if necessary, we may identify $\Sigma = Reg(\Sigma)$ for simplicity. We say $\Sigma$ is \textbf{regular}, if $Sing(\Sigma)=\emptyset$. In addition, we will always assume $\Sigma$ has \textbf{optimal regularity}: $\cH^{n-2} (Sing(\Sigma)) = 0$ and $\cH^n\mres{\Sigma}$ is locally finite.
    \end{Definition}
    
    Similar to the (local) area-minimizing hypersurfaces, for $2\le n\le 7$, it is a well-known result that $Sing(\O)=\emptyset$ if $\O$ is an isoperimetric region \cite{almgren1976existence,gonzalez1983regularity,gruter1987boundary}. So higher dimension is the only possible case that singularities may appear. On the other hand, Simons' cone \cite{bombieri1969minimal} manifests an example of an area-minimizing hypersurface in Euclidean space, but isoperimetric regions are smooth regardless of dimensions. 
    \begin{Example}
        [Barbosa, do Carmo, and Eschenburg \cite{barbosa1988stability}] Let $M^{n+1}(c)$ denote the simply connected complete Riemannian manifold with constant sectional curvature $c$. Let $X:\Sigma^n \to M^{n+1}$ be an immersion of a differentiable manifold $\Sigma^n$. Suppose $X(\Sigma)$ has constant mean curvature. Then the immersion $X$ is volume-preserving stable (see Definition \ref{Definition: volume preserving stable}) if and only if $X(\Sigma) \subset M(c)$ is a geodesic sphere.
    \end{Example}
    
    Similar to the area functional for the minimal surfaces, we define the following functional for isoperimetric regions. An important corollary is that all isoperimetric regions have constant mean curvature. The above examples are isoperimetric regions in space forms (see also \cite{schmidt1943beweis,kwong2022some}). 
    
    Fixing a metric $g$, we can consider the functional $\mathcal{F}: \R \times \cC(M)\to \R$ by
    $$\mathcal{F}(\lambda, \O)=\P(\O)+\lambda \abs{\O}.$$ 
    Suppose $X$ the smooth compactly supported vector field, $\{\phi_t\}$ the corresponding diffeomorphism, $H$ the generalized mean curvature vector, and $\upsilon$ the outer normal vector for $\p^* \O$. By the area formula and first variation for potential energy (e.g., \cite{maggi2012sets} Chapter 17), we can get the first variation with respect to $\O$:
    \begin{align*}\label{Isoperimetric functional}
        \d_2\F(\lambda, \O)&=\res{\frac{d}{dt} \F(\lambda,\phi_t(\O))}{t=0}\\
        &=\int_{\p^* \O} \div X \, d\H^{n-1}+ \lambda \int_{\p^*\O} X\cdot \upsilon \, d\H^{n+1}.
    \end{align*}

    \begin{Remark}
        In addition, if $\O$ is the isoperimetric region (i.e. the minimizer of \eqref{minimizer}), we have $ \d_2\F(\lambda, \O)=0$. As remarked in \cite{ilmanen1996strong, schoen1981regularity}, by using the estimates of Hausdorff dimension for the singular set, $\cH^{n-2} (\p \O)=0$, then using a cutoff function argument for the test functions, we can see that $\lambda=-H\cdot \upsilon=-h$ for any points in $Reg (\p \O)$, which provides $Reg (\p \O)$ a constant mean curvature hypersurface in $M$. We call $\Sigma$ a $H$-hypersurface if $H_{\Sigma} \equiv H$.
    \end{Remark}
    
    Similar to the area-minimizing hypersurfaces, isoperimetric regions have the following sense of ``stability''.
    \begin{Definition} \label{Definition: volume preserving stable}
        For $\O \in\cC(M)$, we say $\O$ is volume-preserving stable if $\p^* \O$ has constant mean curvature and for any diffeomorphism $\phi$ with $|\phi(\O)|_g=|\O|_g$, we have $\d^2\P(\O)\ge 0$.
    \end{Definition}

    Moreover, suppose $\p\O$ is smooth. By \cite[Proposition 2.5]{barbosa1988stability}, suppose 
        $$f(x)=\Prod{\frac{\p\phi}{\p t}(x)}{\nu(x)},$$ 
    we have 
    \begin{align*}
        \d^2\P(\O)&=\d^2_2\F (-H,\O) \\
        &=\int_{\p \O} -f\Delta f- (|A_{\Sigma}|^2 + Ric_M(\nu, \nu))f^2 \, d\cH^n(x)\ge 0,
    \end{align*}
    where $\int_{\p \O} f =0. $

    \begin{Remark}
        For the sake of disambiguation, we will say 
        ``stable''
        if $\Sigma$ is a minimal surface such that $\d^2(\Sigma)>0$ for any diffeomorphisms. 
    \end{Remark}
    
    Using the definition in \cite{smale1999singular}, we define ``strictly stability'' as follows, which generalized the definition over isolated cones (see \cite{caffarelli1984minimal}).
   
    \begin{Definition}\label{Definition: Strictly stable}
        We say $\Sigma$ is \textbf{strictly stable} if $\Sigma$ is a minimal surface, with singularities on a closed set of codimension 6, there is a positive constant $C$ such that
        $$ \int_\Sigma |\nabla u|^2 - (|A|^2 +Ric(v)  )u^2 \ge C\int_\Sigma u^2 \rho^{-2} , $$
        for all $u \in W_0^{1,2,-1} (\Sigma).$
        Here we define $\rho (x) =dist (x,Sing(\Sigma))$ and $W_0^{1,2,-2}(\Sigma) = \overline{ C_0^1(\Sigma\setminus Sing(\Sigma)) }$ with norm
        $$  \norm{u}^2_{1,2}:= \int_\Sigma |\nabla u|^2+u^2\rho^{-2}. $$
   \end{Definition}
    \begin{Remark}
        Smale in \cite{smale1999singular} shows that after a conformal change of metric, the hypersurface constructed in $\mathbb S^8$ is strictly stable in the sense of Definition~\ref{Definition: Strictly stable}, which implies the existence of a neighborhood such that it is homological minimizing in it (see Theorem~\ref{Lemma: Local homological minimizing} or \cite[Lemma 4]{smale1999singular}).
   \end{Remark}

\subsection*{Cones with isolated singularities}
\hfill

    Because the tangent cones of isoperimetric regions are area-minimizing cones, we will exhibit some properties about cones with isolated singularities, i.e., singular minimizing cones $K\subset \R^{n+1}$ with $Sing(K)=\{0\}$ and $n\ge 7$. 

    \begin{Example}[Minimal cones]
        Given $\Sigma^{n-2} \subset \mathbb S^{n-1}$ a smooth minimal hypersurface, the \textbf{cone} based on $\Sigma$ is defined as
        \begin{align}
            {\bf{C}}(\Sigma)=\{ \lambda x: \: x\in \Sigma, \ \lambda> 0 \}. 
        \end{align}
    \end{Example}
    
    In addition, if a cone ${\bf{C}}$ is a minimal surface, then ${\bf{C}}\cap \mathbb S^{n-1}$ is a minimal surface in $\mathbb S^{n-1}$. We denote ${\bf{C}}_1:={\bf{C}}\cap \mathbb S^{n-1}$.

    \begin{Definition}[{\cite[section 3]{hardt1985area}}]
        Let ${\textbf{C}}$ be a regular hypercone (i.e. $Sing ({\bf C})\subset \{0\}$), we say $\textbf{C}$ is strictly minimizing if there is an $\theta > 0$ such that
        $$ \M({\textbf{C}}_1)\le \M(S) - \theta\e^n , $$
        whenever $\e >0$ and $S$ is an integer multiplicity current with $\spt S \Subset \R^{n+1}\setminus B_\epsilon$ and
        $\p S= \p {\textbf{C}}_1$.
    \end{Definition}  
     Hardt and Simon in \cite[Theorem 3.2]{hardt1985area} exhibit several equivalent definitions of strictly minimizing. The critical property of strictly minimizing hypercone is the following property of graphical local minimizing.
    
    \begin{Theorem}[{\cite[Theorem 4.4]{hardt1985area}}] Suppose $\bf C$ be a regular strictly minimizing (multiplicity one) hypercone in $\R^{n+1}$. Let $M$ be a smooth oriented embedded hypersurface in $\R^{n+1}$ with the representation in the form 
        $$ M = graph_{{\textbf{C}}_1} h = \{x+h(x)\nu(x) : \  x\in {\bf C}\setminus\{0\}\}, $$
        where $\nu$ is the unit normal vector and
        $h$ is some function in $C^2({\textbf{C}}_1)$ such that
        $$ \Big| \frac{h(r\o)}{r} \Big| + |Dh(r\o)| \le C r^q, \ \ \ r\o\in {\textbf{C}}_1, $$
        for some $q>0$.
        
        Assume that $\R^{n+1}$ equipped a $C^3$ Riemannian metric $\dis g = \sum_{i,j =1}^{n+1} g_{ij} dx^idx^j$, which satisfies 
        $$g_{ij}(0)=\delta_{ij}, \frac{\p g_{ij}}{\p x_k} (0) = 0,  \  \ i,j,k = 1,\dots ,n+1.$$
        Suppose $M$ is a minimal surface (mean curvature zero) in $(\R^{n+1},g)$, then there is a $\rho>0$ such that $M$ is area minimizing in $B_\rho(0)$ with respect to the metric $g$.
    \end{Theorem}
    \begin{Remark}
    \hfill
    
    \begin{itemize}   
        \item In \cite{smale1999singular}, Smale generalized the above local minimizing property in the manifold setting (Theorem~\ref{Lemma: Local homological minimizing} below), with additionally assuming the cone $\textbf{C}$ is strictly stable. Moreover, Zhihan Wang in \cite[Theorem 5.1]{wang2020deformations} also presents a proof of local minimizing property by constructing a foliation over $\Sigma$.
        \item Simons' cone $\textbf{C}$ is regular, strictly stable, and strictly minimizing. And it splits $\R^8$ into two parts $E_+$ and $E_-$. By \cite[Theorem 2.1]{hardt1985area}, each $E_+$, $E_-$ contains one smooth area minimizing hypersurface up to scaling, call them $R_+,R_-$ respectively. By the symmetry of $\textbf{C}$, $R_+,R_-$ are diffeomorphic to each other. So Simons' cone satisfies the requirement in Remark \ref{Remark: for the main theorem}. In the rest of this paper, we will not use the definitions above about cones with singularities and sufficiently use the Simons' cone.
    \end{itemize}
    \end{Remark}

\section{Constructing the manifolds with Smale's ideas}\label{section: Constructing the manifolds}

This section is dedicated to constructing a family of manifolds, a suitable choice of which will later give us the main Theorem~\ref{Theorem: main theorem}.
We divide it into two parts: first, using a result of Smale \cite{smale1999singular}, we obtain the first piece of our manifold, then we suitably modify it to glue it to a cylinder to obtain the desired construction. The second part is more similar to an 8-dimensional torus example in \cite[Lemma 4.1]{chodosh2022riemannian} by Chodosh, Engelstein, and Spolaor.

\subsection{Smale's main result} \label{subection: Smale's result}
\hfill

    We recall here the main result from \cite{smale1999singular}, which will be the starting point of our construction.

    \begin{Theorem}[{\cite[Lemma 4]{smale1999singular}}]\label{Lemma: Local homological minimizing}
        Let $\bf C$ be any strictly stable and strictly minimizing cone, e.g., Simons' cone. Let $\Sigma:=({\bf{C}}\times \R)\cap \mathbb S^8$. There exists a $C^\infty$ metric $g$ on $\mathbb S^8$ and $\delta>0$ such that $\Sigma$ is uniquely homologically area-minimizing in the tubular neighborhood
        $$
        U_\delta:= \{ x\in \mathbb S^8 : d_{\mathbb S^8}(x,\Sigma)\le \d \},
        $$ with respect to the metric $g$.
    \end{Theorem}

    \begin{proof} The proof of this result can be found in \cite[Lemma 4]{smale1999singular}. 
    \end{proof}

    In order to glue the $U_\delta$ along the boundaries with a manifold (with boundary), we need the boundary of $U_\d$ to be smooth. Because of the singularities of $\Sigma$, we cannot expect the smoothness of the boundaries of $U_\d$ for any small $\d>0$. Fortunately, as remarked in the proof of \cite[Lemma 4]{smale1999singular}, we can find a smaller neighborhood $V\subset U_\d$ such that $\Sigma$ is still homological area-minimizing in $V$, and the boundary $\p V$ is smooth. The basic idea is to glue in pieces of foliations.

    \begin{Theorem}\label{Theorem: V}
         Denote $\bf C$ the Simons' cone. Let $\Sigma:=({\bf{C}}\times \R)\cap \mathbb S^8$ and let $U_\delta$ and g be as in the previous theorem. There exists an open subset $V$ such that $\overline V\Subset U_\delta$ and
         \begin{itemize}
             \item $\Sigma \subset V$ and it is homologically minimizing in $V$ with respect to the metric $g$;
             \item $V\setminus \Sigma$ consist of $2$ connected components, $V_\pm$ and $\partial V_\pm=\Sigma \cup \Gamma_\pm$, disjoint union, and $\Gamma_\pm$ are smooth and diffeomorphic to each other.
         \end{itemize}
         
    \end{Theorem}
    
    \begin{proof}
    The proof is almost the same as the argument in \cite{smale1999singular}.

    For $\d$ small, $\Sigma$ splits $U_\d$ into two parts, call them $U_+,U_-$. Note that $Sing(\Sigma) = \{ p_+, p_- \}$. For $\sigma>0$ smaller than $\d/8$, we have $B(p_\pm, \sigma) \subset U_\d$. Consider the Fermi's coordinate $$\{(x,t) : x\in \Sigma, \ \   \rho(x)> \sigma, \ \ |t|<\d \},$$ 
    around $\Sigma$, here we denote $ \rho(x) :=d_{\mathbb S^8}(x,p_+\cup p_-) $. We denote $\Sigma_\sigma = \{ q \in \Sigma: \rho(q) < \sigma \}$, then consider the constant graph on $\Sigma \setminus \Sigma_\sigma$:
    $$\Gamma_t = graph_{\Sigma \setminus \Sigma_\sigma} t,$$
    with $|t|< \d$.
    
    Let $p=p_\pm$ and denote $ S_t = \{ (x,t): \rho(x,0) = 4\sigma \}$. By \cite[Theorem 2.1, 5.6]{hardt1985area}, $S_t$ bounds a smooth $7$-dimensional submanifold $R_t$ which is  area minimizing in $B_p(5\sigma)$, and as $t\to 0$, $R_t\to\Sigma_{4\sigma} $ in the current and Hausdorff sense. And there exists $\e_t>0$, a $C^2$ function $u_t$ on $\Sigma_{4\sigma} \setminus \Sigma_{\e_t}$ such that $R_t \setminus B(p,\e_{t})$ can be described as a graph of $u_t$.

    On the other hand, we can find a smooth cutoff function $\chi$ supported on the annulus $\Sigma_{4\sigma} \setminus \Sigma_{\e_t}$ such that 
    \[
    \chi(x) = 
    \begin{cases}
        1 & \text{for } \rho(x) \ge 3\sigma , \\
        0 & \text{for } \rho(x) \le 2\sigma .
    \end{cases}
    \]
    
    We can get a smooth hypersurface by gluing the smooth submanifold $\Gamma_t$ with $R_t$ through a function $w$ on $\Sigma \setminus\Sigma_{\e_t}$ by
    $$ w = \chi u_t + (1-\chi) t. $$

    Therefore, for each side of $\Sigma$, we can employ any small positive (resp. negative) $t$ to find a smooth hypersurface, denoted by $\Gamma_+$ (resp. $\Gamma_-$) in $ U_+$ (resp. $U_-$), which is homologous to $\Sigma$. 
    Denote $V$ the neighborhood of $\Sigma$ bounded by $\Gamma_+, \Gamma_-$, $V$ clearly has a smooth boundary. Moreover, note that $\textbf{C}$ splits $\R^8$ into two parts, by the symmetry of the Simons' cone, $R_t$ is diffeomorphic to $R_{-t}$ for all small $t$. Therefore, $\Gamma_+,\Gamma_-$ are diffeomorphic to each other.
    
    In summary, $\Sigma$ splits $V$ into two parts: $V_+,V_-$. And $\p V_+$ has two components: the smooth part $\Gamma_+$, and $\Sigma$ (similarly for $\p V_-$). Then we have $\p V= \Gamma_+ \cup \Gamma_-$, where $\Gamma_+,\Gamma_-$ are both smooth and homologous to $\Sigma$. Moreover, $\Gamma_+$ is diffeomorphic to $\Gamma_-$. We will use the set $V$ in the next subsection to construct the manifolds with singular isoperimetric regions.
    \end{proof}

\subsection{Construction  of the toric manifolds} \label{Subsection: construction of manifolds}
    \hfill
    
    Next, we will construct a collection of $8$-dimensional closed Riemannian manifolds, which we will use in the next section. Again we denote $V$, the smooth neighborhood of $\Sigma$ from the last subsection.

    \begin{Theorem}\label{Theorem: Get M(R)}
         Denote $\Gamma$ the $7$-manifold which is diffeomorphic to $\Gamma_+$ and $\Gamma_-$ by $F_\pm :\Gamma_\pm\to\Gamma $. Consider the smooth manifold defined by gluing the boundaries of $V$ and $\Gamma\times [0,R]$:
         \begin{align}
             M(R):=\Gamma \times [0,R]\cup V/\Gamma_{\pm}\times \{0,R\} \sim_{F_\pm} \tilde \Gamma_{\pm}.
         \end{align}     
         
         There exist $R_0>0$ and a one parameter family of $C^\infty$-metrics $(g_R)_{R>0}$ on $M(R)$ such that, if $V$ is as in Theorem~\ref{Theorem: V}, then for every $R>R_0$, the followings hold:
         \begin{enumerate}
             \item $V$ is isometrically embedded into $(M(R), g_R)$;
             \item $\Sigma$ is the unique homological area minimizer in $(M(R),g_R)$.
         \end{enumerate}
    \end{Theorem}

    The idea is the following: suppose $g_{l}$ is some Riemannian metric of $\Gamma$ to be determined. At first, we conformally change (enlarge) the metric $ g$ of $V$ in $U_\d$ near $\Gamma_+,\Gamma_-$ such that the outside of a tubular neighborhood of $\Gamma_+$ (and $\Gamma_-$) is isometric to the cylindrical metric $(\Gamma \times [0,\e],g_l+dr^2)$ for sufficiently small $\e$, . Additionally, we require that $\Sigma$ is still homological area-minimizing in $(V,\wtilde g)$, where $\wtilde g$ is the metric after the conformal change.

    Finally, we glue $(V,\wtilde g)$ with $(\Gamma\times [0, R], g_l +dr^2)$ along the boundaries respectively.
    Denote $$ \tilde \Gamma_+ := \Gamma\times \{ 0\}, \qquad \qquad  \tilde \Gamma_- := \Gamma \times \{ 1\} .$$ There is a Riemannian metric on $M(R)$ depending on the length $R$. And we denote 
    $$(M(R),g_R):=\Gamma \times [0,R]\cup V/\Gamma_{\pm} \sim \tilde \Gamma_{\pm}.$$

    \begin{proof}[Proof of Theorem~\ref{Theorem: Get M(R)}]

    For each $i=\pm$, consider $U_i(\supset \Gamma_i)$ the Fermi's coordinate $(x,t)$ for $x\in \Gamma_{i}$, here choosing $t$ with the positive direction towards outside of $\Sigma$, then there exists an $\e>0$ such that 

    \begin{itemize}
        \item[(1)] for $|t|\le \e$, we have $(x,t) \in U_{i}$ respectively;
        \item[(2)] there exists an $\e_1>0$ such that
        $$\inf_{|t|\le \e}\inf_{ x\in \Gamma_{\pm} } 
        dist((x,t), \Sigma )>\e_1.$$
    \end{itemize}
    
    For each small $t$, we define
    $$\Gamma_{i}(t):= \{ (x,t):\: x\in \Gamma_i \}, $$
    again here we define the positive signs for each $\Gamma_\pm$ the side opposite with $\Sigma$, so each $\Gamma_i(t)$ forms a layer in $U_i$.
    
    Note that for any $|t|\le \e$, $\Gamma_+(t),\Gamma_-(t)$ bound an open neighborhood of $\Sigma$ as well. For $t= -\e$, we denote $V_0$ the neighborhood of $\Sigma$ with boundary $\Gamma_+(-\e) \cup \Gamma_- (-\e) $. In another word, $\Gamma_+(-\e), \Gamma_- (-\e) $ split the neighborhood $U_\e$ into three parts, and $\Sigma$ lies in the middle part. 
    
    For $-\e < \sigma\le \e$, denote
    \begin{align*}
        V_+(\sigma)&:= \{(x,t):\: x\in \Gamma_+ , \ -\e\le t<\sigma \};\\
        V_-(\sigma)&:= \{(x,t):\: x\in \Gamma_- , \ -\e\le t<\sigma \}.
    \end{align*}
    
    We can denote the neighborhood (depending on $\sigma$) by 
    \begin{align}\label{equation: V(sigma)}
        V(\sigma):= V_0 \cup V_+(\sigma) \cup V_-(\sigma). 
    \end{align}

    Note that for each small $\sigma$, $V(\sigma)$ is a neighborhood of $\Sigma$ with two smooth boundaries $\Gamma_i(t)$ such that each $\Gamma(t)$ is homological to $\Sigma$. In addition, by Theorem~\ref{Lemma: Local homological minimizing}, $\Sigma$ is uniquely homological area-minimizing in $V(\sigma)$ for any $-\e \le \sigma \le \e$.

    In addition, we observe that $V : = V(0).$ Following the ideas at the beginning of this section, we will glue the smooth neighborhood $V(\e)$ along the boundary with a manifold (with boundary) endowed with a cylindrical metric.

    Now, we will keep the metric on $V_0$ and deform the metrics on $V_\pm(\e)$ to the cylindrical metric (near $\Gamma_{\pm}(\e)$). Note that for each $i=\pm$, under the Fermi coordinates in $V_i(\e)$, the metric has the form:
    $$ g(x,t) = g_i (x,t) + \eta_i(x,t)dt^2 ,$$
    where for each small $t$, $g_i(\cdot, t)$ is a metric on $\Gamma_i$ respectively, and $\eta_i$ is a smooth positive function.

    In order to construct a deformation of the metric on $V(\e)$, we first define new metrics on $\Gamma_+(\e),\Gamma_-(\e)$. Note that $\Gamma_+(\e),\Gamma_-(\e)$ are both diffeomorphic to $\Gamma$. Consider $g_l$ the Riemannian metric on $\Gamma$ and the diffeomorphisms $F_\pm$:
    $$ F_+ : \Gamma_+(\e) \longrightarrow \Gamma(=(\Gamma,g_l)), \ \ \ \ F_- : \Gamma_-(\e)  \longrightarrow \Gamma,$$
    such that for any $x\in \Gamma_\pm (\e)$ and any $v\in T_x \Gamma_\pm(\e)$, we have
    $$ F_\pm^* g_l(v,v) \geq 2 g(v,v), $$ 
    where $g$ denotes the original metric in $V(\e)$.

    To construct a new metric on $V(\e)$, for $i = \pm$, we define the metric $\overline{g}_i$ on $\Gamma_i$ and the number $\overline \eta_i$ by

    \begin{align}\label{equations:enlarge metric}
        \overline g_i (x) &:= F^*_i (g_l) (x),\\
        \overline \eta_i &:= \max_{\overline V_i(\e)} \eta_i(x,t).
    \end{align}
    
    Denote the cutoff function $\phi:\R \to \R$ a smooth non-negative function such that
    \[
    \phi(x)=
    \begin{cases}
        0 & t\le 0;\\
        1 & t\ge \e/2,
    \end{cases}
    \]
    
    Then we can define new metrics $\wtilde g_i$ on $V_i(\e)$, where $i=+$ or $-$, by the following:

    \begin{align}\label{equation: interpolation metric}
        \wtilde g_i(x,t)=&(1-\phi(t))g_i(x,t) +\phi(t) \overline g_i(x)\\
        &+[(1-\phi(t))\eta_i(x,t) +\phi(t)\overline \eta_i] dt^2. \nonumber
    \end{align}
    So clearly, each of\, $\wtilde g_+$ and $ \wtilde g_-$ is simply an interpolation such that, as $t$ increases from $0$ to $\e/2$, the original metric deforms to the cylindrical metric. Moreover, we leave the metric unchanged when $t$ is non-positive.

    Now we can define a new metric $\wtilde g$ on the whole neighborhood $V(\e)$ by the following:
    \[
    \wtilde g(p)=
    \begin{cases}
        g(p), & p\in V_0 , \\
        \wtilde g_i(p), & p\in V_i(\e),\ i=+ \ \text{or}\ -.
    \end{cases}
    \]

    Therefore, we can construct the collection of manifolds $\{M(R)\}_R$. Note that $V(\e)$ has cylindrical metric on $V(\e)\setminus V(\e/2)$, and $\Gamma_\pm(\e)$ are both isometric to $\Gamma$. Then consider the maps $F_\pm:\Gamma_\pm\to \Gamma$ defined above.
    Denote $W(R)= [0,R] \times \Gamma$ for $R>0$ a Riemannian manifold with the boundary equipped with the product metric. Along with the diffeomorphisms $F_\pm$, we glue $\Gamma_+(\e)$ with $\{0\} \times \Gamma$ and glue $\Gamma_-(\e)$ with $\{R\} \times \Gamma$. Then we get a connected smooth Riemannian manifold depending on the positive number $R$, denote it by $(M(R), g_R)$.

    Finally, note that under the Riemannian metric $\wtilde g$, \eqref{equations:enlarge metric}-\eqref{equation: interpolation metric} show that we leave the metric in $V$ unchanged and enlarge the metric on $V_+(\e),V_-(\e)$. So $\Sigma$ is still the unique homological area minimizer in $ V(\e)$. Moreover, by \cite[Lemma 5]{smale1999singular}, we see that for sufficiently large $R>0$, under the metric defined above, $\Sigma$ is the unique homological area minimizer in $M(R)$. So we construct a Riemannian metric of $M(R)$ such that $(V,\wtilde g)$ (the smoothed neighborhood from Theorem~\ref{Theorem: V}) isometrically embedded into $M(R)$.

    \end{proof}

\section{Proof of the main Theorem~\ref{Theorem: main theorem}}\label{Section: Proof of the main theorem}

    In section~\ref{Subsection: construction of manifolds}, we have constructed a collection of closed manifolds that we need later as ambient spaces. Then, in section~\ref{subsection: Construction of singular isoperimetric region}, we will construct the singular isoperimetric region in the corresponding manifold. In this section, we will redefine $V$ for convenience by
    \begin{align}\label{equation: redefine V}
        V := V(\e)
    \end{align}
    where $V(\e)$ is defined in \eqref{equation: V(sigma)}.

\subsection{Construction of singular isoperimetric region}\label{subsection: Construction of singular isoperimetric region}
\hfill
\medskip

    \begin{Theorem}\label{Theorem: exists singular region}
        Let $M(R)$ be the family of manifolds constructed in Theorem~\ref{Theorem: Get M(R)}. There exists $R_1>0$ such that for any $R>R_1$, there is $t_0\in (0,R)$ such that the boundary of the two unique isoperimetric regions (the one and its complement) with volume $|M(R)|_g/2$ is of the form
        $$\Sigma \cup [\Gamma\times \{t_0\}]. $$
    \end{Theorem}
    This theorem directly implies Theorem~\ref{Theorem: main theorem}.

    \begin{Example}
        Consider the torus $\mathbb S^7\times \mathbb S^1(R)$ with product metric for large length $R$. \cite{chodosh2022riemannian} shows that for large $R$, the boundary of an isoperimetric region with half volume is of the form 
        $$\mathbb S^7 \times \{0\} \cup \mathbb S^7 \times \{\pi R\}. $$
    \end{Example}

    To prove Theorem~\ref{Theorem: exists singular region}, we want to have a uniform bound on the mean curvature of isoperimetric regions with volume bounded away from zero and that of the manifold. For a closed manifold with dimension $2\le n\le 7$, this lemma is proved in \cite{chodosh2022riemannian}. In the higher dimensional cases, the boundary can have singularities; we want to have the mean curvature bounded in $Reg(\p \O)$. Adapting the argument of \cite[Lemma C.1]{chodosh2022riemannian} shows the following lemma:

    \begin{Lemma}[\textbf{cf.} {\cite[Lemma C.1]{chodosh2022riemannian}} and \cite{ros2010stable,cheung1991non}] \label{Lem: curvature control for isoperimetric regions}
        For $ n \ge 2$, fix $\delta>0$, and $(M^{n},g)$ a closed Riemannian manifold with $C^3$-metric, there is $C = C(M,g,\delta) <\infty$ so that if $\O\in \cC(M, t)$ is an isoperimetric region with volume $|\O|_g \in (\delta, |M|_g - \delta)$, then the mean curvature of $Reg(\p \O)$ satisfies $|H|\le C$.
    \end{Lemma}
    \begin{proof}
        \cite[Lemma C.1]{chodosh2022riemannian} proves the case $2\le n\le 7$. Next, we assume $n\ge 8$, the proof is similar to \cite[Lemma C.1]{chodosh2022riemannian}.
        
        Assuming not, we would have a sequence of isoperimetric regions $\O_j \subset (M,g)$ with $|\O_j|_g \in (\d, |M|_g -\d)$ with divergent constant mean curvature $H_j$ such that $\lambda_j :=| H_j| \to \infty$.
        
        Choosing any $x_j \in \p \O_j$, we can rescale the metric in $\lambda_j$ at the point $x_j$, denote $\wtilde g_j = \lambda_j^2g $. So we have $(M, \wtilde g_j, x_j)  $ converges in $C^3_{loc}$ to the flat metric on $(\R^n,0)$. Also, for each $j$, we have $\wtilde \O_j$ an isoperimetric region in $(M, \wtilde g_j)$. Passing to a subsequence, there is a locally isoperimetric region\footnote{ We say $\O$ is a locally isoperimetric region if for any $R>0$ and $\wtilde \O $ with $\O\Delta \wtilde \O \Subset B_R$ and $|\O \cap B_R| = |\wtilde \O \cap B_R|$, we have $\P(\O,B_R) \le \P(\wtilde \O, B_R)$. } $\wtilde{\O}$ in $\R^n$ such that $\wtilde \O_j$ locally converges to $\wtilde \O$ in the varifold sense. 
        In addition, for any $p \in Reg(\p \wtilde \O)$, there exists a neighborhood $B(p,r)$ around $p$ for some $r>0$ such that $\p \wtilde\O_j \cap B(p,r) \subset Reg(\p \wtilde \O_j)$, and therefore $Reg(\p \wtilde \O_j)\cap B(p,r)$ converges in $C^{2,\a} $ to $Reg(\p \wtilde \O) \cap B(p,r)$. Therefore, the mean curvature of $\wtilde \O$ is $\pm 1$. Moreover, $\wtilde \O$ is volume-preserving stable. 
        
        Next, we claim that $\wtilde \O$ is non-compact. Suppose not, say $ \wtilde \O$ is compact for any choices of $x_j \in \p \O_j$. Then $\wtilde \O$ is compact and has to be an isoperimetric region in $\R^n$. By \cite[Theorem 14.1]{maggi2012sets}, it has to be a ball. By the mean curvature $|\wtilde{H}|=1$, $\wtilde{\O}$ is a unit ball. So $\O_j$ will be regular too. Note that $(\O_j)_j$ has a uniform lower and upper bound of volume, as $\lambda_j \to \infty$, the volume of $(\wtilde\O_j)_j$ in $\wtilde{g}_j$ will converge to infinity. So $(\p\wtilde\O_j,\wtilde{g}_j)$ would be close to a union of an increasing number of regions close to the geodesic spheres. Therefore, for $\lambda_j$ sufficiently large, $\frac{1}{2\lambda_j}$ will be the lower bound of the diameters of the balls of $\O_j$. Denote $V>0$ the lower bound of the volumes of $(\O_j)_j$. So for any large $j$, it at least has $V\lambda_j^{n}$ many balls. So we have
        $$ \P(\O_j) \ge C(n)\frac{V}{\lambda_j^{n-1}} \P(B_{\lambda_j}) \ge  C(n)\frac{V}{\lambda_j^{n-1}} \lambda_j^{n}. $$
        Therefore, as $\lambda_j \to \infty$, $\P(\O_j) \to \infty$, contradiction. Therefore, we can assume $\p \wtilde \O$ is non-compact.

        \medskip
        Denote $\wtilde H:=$ the mean curvature of $Reg(\p\wtilde \O)$, and $\wtilde A:=$ the second fundamental form of $Reg(\p\wtilde \O)$. Because $|\wtilde H| =1 $, we have $ |\wtilde A|^2 \ge \frac1n$. Therefore, for any $\phi\in C^1_c(Reg(\p \wtilde \O ) )$ with $\int_{Reg(\p \wtilde \O )} \phi \, d\cH^{n-1} = 0$, we have
        $$ \int_{Reg(\p \wtilde \O )}  \phi^2 \, d\cH^{n-1} \le  \int_{Reg(\p \wtilde \O )} n|\wtilde A|^2 \phi^2 \, d\cH^{n-1}
        \le
        n\int_{Reg(\p \wtilde \O )} |\nabla\phi|^2 \, d\cH^{n-1}.  $$
        Following the remark from \cite{fischer1985complete} and \cite[Proposition 2.2]{barbosa2000eigenvalue}, we see that $\p \wtilde \O$ is \textbf{strongly} stable outside of a compact set, i.e., there exists a large enough $R>0 $ such that for any $\phi\in C^1_c(Reg(\p \wtilde \O ) \setminus B_R )$, we have
        \begin{align}\label{equation: non compact stability inequlity}
            \int_{Reg(\p \wtilde \O )}  \phi^2 \, d\cH^{n-1} \le n \int_{Reg(\p \wtilde \O )} |\nabla \phi|^2 \, d\cH^{n-1}.
        \end{align}

        Consider the $C^1$ radial cutoff function $\phi$ in $\R^n$ such that: fixing any $\rho>R+1$,
        \[
         \phi(x)=
        \begin{cases}
            1 & |x|\in [R+1, \rho] ,\\
            0 & |x|\in [0, R] \cup [2\rho, \infty].
        \end{cases}
        \]
        with $ |D\phi(x)| \le C(n) \rho^{-1}$, here $D$ is the Euclidean connection on $\R^n$. Note that $\cH^{n-3}(Sing (\p \wtilde\O)) = 0$. Given any $\e>0$, consider $\{B_{r_j}(p_j)\}_j$ a collection of geodesic balls which cover $Sing(\p \wtilde\O)$, such that 
        $$ \sum_j r_j^{n-3} <\e. $$
        
        We define $\psi_j$ a smooth cutoff function by 
        \[
         \psi_j (x)=
        \begin{cases}
            1 &  \text{if } x\notin B_{2r_j} (p_j), \\
            0 & \text{if } x\in B_{r_j} (p_j).
        \end{cases}
        \]
        with $|D \psi_j|\le C(n) r_j^{-1}$. Now we define 
        \begin{align*}
           \psi_\e &:= \inf_j \psi_j \\
           \Phi_\e &:=  (\phi)^{\frac{n-1}{2}} \cdot \psi_\e.
        \end{align*}
         
        So we note that $\Phi_\e$ is a Lipschitz compactly supported function on $Reg(\p \wtilde \O)$. So by \eqref{equation: non compact stability inequlity}, we have
        \begin{align*}
            &\int_{Reg(\p \wtilde \O )}  \Phi_\e^{2} \, d\cH^{n-1} \le n\int_{Reg(\p \wtilde \O )}  |\nabla\Phi_\e|^{2} \, d\cH^{n-1}\\
            & =2n \int_{Reg(\p \wtilde \O )}  |D\phi^{\frac{n-1}{2}} |^2 \cdot \psi_\e^2 + \phi^{n-1} \cdot |D \psi_\e|^{2} \, d\cH^{n-1}\\
            & =  2n \int_{Reg(\p \wtilde \O )} \left(\frac{n-1}{2}\right)^2 \phi^{n-3} |D\phi |^2 \cdot \psi_\e^2 + \phi^{n-1} \cdot |D \psi_\e|^{2} \, d\cH^{n-1}\\
            & =  \int_{Reg(\p \wtilde \O )} \left( \frac{n-3}{n-1}\cdot\phi^{n-1} +C(n)  |D\phi |^{n-1} \right)   \cdot \psi_\e^2 \, d\cH^{n-1} + 2n\int_{Reg(\p \wtilde \O )}  \phi^{n-1} \cdot |D \psi_\e|^{2} \, d\cH^{n-1}.
        \end{align*}
        Therefore, abusing the notations of constants $C(n)$, we have
        \begin{align}\label{equation: perimeter cutoff}
            \int_{Reg(\p \wtilde \O )}  \Phi_\e^{2} \, d\cH^{n-1} \le C(n) \int_{Reg(\p \wtilde \O )}  |D\phi |^{n-1}   \cdot \psi_\e^2 \, d\cH^{n-1} + C(n)\int_{Reg(\p \wtilde \O )}  \phi^{n-1} \cdot |D \psi_\e|^{2} \, d\cH^{n-1}.
        \end{align}
        
        For the first part of the right hand side of \eqref{equation: perimeter cutoff}, by the definition of $\phi,\psi_\e$, we have
        \begin{align*}
            &\int_{Reg(\p \wtilde \O )}  |D\phi |^{n-1}   \cdot \psi_\e^2 \, d\cH^{n-1} \le  \rho^{1-n} \cH^{n-1} (\p \wtilde \O \cap B_{2\rho}).
        \end{align*}
            
        For the second part of \eqref{equation: perimeter cutoff}, we have

        \begin{align*} 
        \int_{Reg(\p \wtilde \O )}  \phi^{n-1} \cdot |D \psi_\e|^{2} \, d\cH^{n-1} &\le \sum_j\int_{Reg(\p \wtilde \O )\cap B_{2r_j}(p_j) }  r_j^{-2} \, d\cH^{n-1}\\
        &\le \sum_j r_j^{-2}  \cdot C r_j^{n-1}\\
        &\le C \e.
        \end{align*}
        Here the volume bound comes from the monotonicity formula \cite[17.6]{simon1983lectures}, i.e., we have $\cH^{n-1} (Reg(\p \wtilde \O ) \cap B_{2r_j} (p_j)) \le Cr_j^{n-1}$ for the constant $C$ depending on $|\wtilde{ H}|$, which is constantly 1 in our case.

        Therefore, as $\e\to 0$, the dominated convergence theorem implies that
        \begin{align}\label{equation: perimeter control}
            \cH^{n-1} (\p \wtilde \O \cap (B_\rho \setminus B_R )) \le C (1+ \rho^{1-n} \cH^{n-1} (\p \wtilde \O \cap B_{2\rho}) ).
        \end{align}
        
        Note that as $\rho \to \infty$, and $\p \wtilde \O$ is not compact, we can cover $\p \wtilde \O$ by a countable collection of geodesic balls $(B_j)_j$ such that the concentric balls in $(B_j)_j$ with the half radius are pairwise disjoint. Then because $|\wtilde H| = 1$, by the monotonicity formula (from below), we have that $$\cH^{n-1} (\p \wtilde \O \cap (B_\rho \setminus B_R )) \to \infty .$$ 
        
        On the other hand, note that $\wtilde \O$ is a locally isoperimetric region. For any $\rho$ large, consider $0<r(\rho)\le \rho $ with the property that 
        $$  |\wtilde \O \cap B_{\rho}| = | (\wtilde \O \setminus B_\rho) \cup B_{r(\rho)}|,$$
        has the same volume, so for any large $\rho$ so we have 
        
        $$  \P(\wtilde \O ; B_{\rho})\le Cr(\rho)^{n-1} +C\rho^{n-1}\le  C \rho^{n-1}.$$
        So we get a contradiction with \eqref{equation: perimeter control}.
    \end{proof}
    
    \begin{Lemma}\label{Lemma:Perimeter,components,diameter} Let $\O_R \in (M(R),g_R)$ the isoperimetric region with volume $|M(R)|/2$, and let $\p \O_R=\bigcup_{i=1}^LI_i$, where $\{I_i\}$ are the connected components of $\p \O_R$. Then there exists a nonnegative constant $R_0=R_0(V)$  (V is defined in \eqref{equation: redefine V}), such that for any $R>R_0$, the following hold:
    \begin{enumerate}
        \item $\P(\O_R) \le  \M(\Gamma) +\M(\Sigma);$
        \item each $I_i$ has uniformly bounded diameter;
        \item there exists an integer $L_0=L_0(V,R_0)>1$ such that the number of connected components satisfies $1<L<L_0$;
        \item $I_1$ and $I_2$ are homologous to $[1]:=[\Gamma] \in H_n(M(R),\Z_2)$.
    \end{enumerate}
    \end{Lemma}
    \begin{proof}
        (1) Note that 
        \begin{align}\label{eqn: perimeter upper bound}
            \P(\O_R)\le \M(\Gamma)+ \M(\Sigma),
        \end{align}
        because there is clearly a $U_R\in C(M(R),g_R ,|M(R)|_{g_R}/2)$ such that $U_R =  V_+ \cup (\Gamma \times [0, t_0])$ for some $t_0$, which has half volume and with the boundary 
        $$ \big[ \{t_0\}\times \Gamma \big] \cup \Sigma. $$
        
        (2) Next suppose that there is a sequence $R_k \to \infty $, denote with $\O_k \subset (M(R_k),g_k)$ the corresponding isoperimetric regions with $|\O_k|_{g_k} =\frac12 |M(R_k)|_{g_k}$. Here we denote $g_k:=g_{R_k}$. Suppose $I_k\in \p \O_{k}$ is a component such that $diam(I_k) \to \infty.$ We observe that all the metrics $g_R$ are locally isometric, so by Lemma~\ref{Lem: curvature control for isoperimetric regions}, we have a uniform bound of mean curvature for $Reg(\p \O_k)$ and all $k\in \N $. Then monotonicity formula \cite{allard1972first, cheung1991non} shows that for any $x\in \p \O_k$, we have
        \begin{align}\label{align: monotonicity}
            f(r) := e^{7 (\norm{H_k}_{\infty} +C) r} \frac{\P( \O_k ;B(x,r))}{ \o_7 r^7},
        \end{align}
        is non-decreasing. Where $C$ depends on the upper bound of sectional curvatures of $M(R_k)$ (so independent with $R_k$), $H_k$ is the mean curvature of $Reg(\p \O_k)$, and $\o_7$ the volume of a unit ball of dimension $7$.
        So the $\diam(I_k) \to \infty$ implies that $\M(I_k) \to \infty $, contradictions with (1).

        (3) Given any sequence $R_k\to \infty$, first, suppose that the isoperimetric regions are connected. Note that by (1) and (2):
        \begin{align}
            \P(\O_k)\le \M(\Gamma)+ \P(\Sigma),
        \end{align}
        so they have a uniform bound of diameters. So for each $k$ large, under a translation of the coordinate along the cylindrical direction, $\p \O_k\subset M(R_k)\setminus [0,\frac34 {R_k}] \times \Gamma $,
        therefore, $|\O_k|$ cannot be equal to $\frac12 |M(R_k)|$. This leads to a contradiction. 

        In addition, by the uniform bound of the mean curvatures for $\O_k$ with all large $R$, the monotonicity formula \eqref{align: monotonicity} implies there exists a $c=c( V)$ such that the perimeter of each component of $\p\O_R $ is bounded below by $c$. By \eqref{eqn: perimeter upper bound}, we get a uniform bound $L$ about the number of components of $\p\O$ for all large $R$.

        (4) Denote $\p\O_R= \bigcup_i I_i$. Suppose (4) fails. Note that by the construction, $M(R)$ is homeomorphic to $\Gamma\times \S^1$, then by the boundedness of diameter, all $\{I_i\}$ are boundaries. Reasoning as in the first part of (3), there exists $(t_i)_i$ and $(T_i)_i$ such that
        \[
        I_i \setminus V\subset Tube(T):=[t_i,t_i+T_i]\times\Gamma,
        \qquad \text{and}\qquad
        T_i<T_0<\infty\,,
        \]
        for every $i=1,\dots, L$, where $T_0$ is independent of $i$ and exists by (2). In particular, since each $I_i$ is a boundary, we have
        \[
         \O_R\subset V\cup \bigcup_{i=1}^L Tube(T_i)\subset V\cup\bigcup_{i=1}^L Tube(T_0).
        \]
        Therefore we conclude
        $$ \Vol (\O_R)\leq \Vol\left(V\cup \bigcup_{i=1}^L Tube(T_i)\right)\le \Vol(V)+ \sum_{i=1}^L T_0\cdot \M(\Gamma). $$
        which for sufficiently large $R$ cannot be half the volume of the whole manifold.  
    \end{proof}
    
    \begin{proof}[Proof of Theorem~\ref{Theorem: exists singular region}]

        We still denote $\O_R$ the isoperimetric region in $M(R)$ with half volume. In order to prove the isoperimetric region is of the form we want, we will consider the connected components of $\p \O_R$ into two types:
        \begin{align*}
            \{ \Lambda_i \} &:=\ \{\Lambda_i \subset\p\O_R : \Lambda_i \cap \overline V = \emptyset \} ;\\
            \{\Delta_i \} &:= \ \{ \Delta_i \subset\p\O_R : \Delta_i \cap \overline V \neq \emptyset \}.
        \end{align*}

        Combining the results in Lemma~\ref{Lemma:Perimeter,components,diameter}, we claim the following result for isoperimetric regions $\p\O_R$ with large $R$. 

        \noindent
        {\bf Claim:} With $R_0$ from Lemma~\ref{Lemma:Perimeter,components,diameter}, there exists $R_1>R_0$, such that for any $R>R_1$, $\p \O_R$ has exactly 2 components $\Lambda,\Delta$, where $ \Lambda \in \{ \Lambda_i \}$ and $\Delta\in\{\Delta_i \}$. Specifically, $[\Lambda]=[\Delta]=[1]$ where $[1]:=[\Sigma] \in H_n(M(R),\Z_2)$.
        
        \noindent
        \textit{Proof of the claim:}
        
        \noindent
        \textbf{Step 1.} $\{\Delta_i\} \neq \emptyset$ and at least one $\Delta_i$ is homologous to $\Sigma$.
        
        Suppose not, then $I_1,I_2\subset [0,T] \times \Gamma$, where $I_1, I_2$ are as in Lemma~\ref{Lemma:Perimeter,components,diameter}(4). Since $[I_1]= [I_2] =[1] \in H_n(M(R),\Z_2)$ and $I_1,I_2\in \{\Lambda_i\}$, this implies that
        \[
        \P(\Omega_R)\geq \M(I_1)+\M(I_2) \geq 2\M(\Gamma) > \M(\Sigma)+\M(\Gamma)\stackrel{(1)}{\geq} \P(\O_R)\,.
        \]
        A contradiction.





        \noindent
        \textbf{Step 2.} $\{\Lambda_i\} \neq \emptyset$ and at least one $\Lambda_i$ is homologous to $\Sigma$.

        Suppose not, then we have that $\partial \Omega_R=\bigcup_{i=1}^{L_1}\Delta_i \cup \bigcup_{i=1}^{L_2}\Lambda_i$, with $L_2=0$ if $\{\Lambda\} = \emptyset $; and in the other case, each $\Lambda_i=\p U_i$, for some open connected $U_i\subset \O_R$. Moreover, by Lemma~\ref{Lemma:Perimeter,components,diameter}(3), $\diam(\Delta_i)<d_0$ and $L_1+L_2<L_0$, with $L_0, d_0$ independent of $R$. Now notice that since $\Delta_i\cap V\neq \emptyset$, there exist $T_0>0$, depending only on $d_0$, such that
        \[
        \Omega_R\subset V\cup [(0, T_0)\cup (R-T_0)] \times\Gamma\cup\bigcup_{i=1}^{L_2}U_i\,,
        \]
        with $\Vol(U_i)\leq T_0\times \Gamma$.This yields
        \[
        \Vol(\Omega_R)\leq (L_2+2) \cdot T_0\cdot \M(\Gamma)\,,
        \]
        which leads a contradiction for $R$ sufficiently large.

        \noindent
        \textbf{Step 3.} $\p\O_R$ has no other components.
        
        We have concluded that there is at least one $\Delta\in \{ \Delta_i \}$ and at least one $\Lambda\in \{ \Lambda_i \}$ such that
        \begin{align*}
            [\Delta]&=[1]\in H_7(M(R), \Z_2) ,\\
            [\Lambda]&=[1]\in H_7(M(R), \Z_2).           
        \end{align*}

        Clearly $\Lambda \subset [0,R]\times \Gamma$ where each slice $\{t\}\times \Gamma$ is a homological area-minimizing in the tube $Tube(R)$. So we directly have
        \begin{align*}
            \M(\Lambda)&\ge \M(\Gamma).
        \end{align*}

        On the other hand, we know that $\Sigma$ is the unique homological area minimizer in $M(R)$. So clearly 
        $$ \M(\Delta) \ge \M (\Sigma). $$

        So overall, by the perimeter upper bound \eqref{eqn: perimeter upper bound}, the only case that would happen is 
        $$\p \O_R = \{t_0\} \times \Gamma \cup \Sigma. $$
    \end{proof}

    We have two isolated singulars on $\Sigma$, so the above lemma proves Theorem~\ref{Theorem: main theorem}.

\section{Proof of Theorem \ref{Theorem: main theorem 2}} \label{Section: Proof of theorem 2}
    The proof of Theorem \ref{Theorem: main theorem 2} strongly relies on the construction of singular homological area minimizers in higher dimensions, i.e.,  the following lemma from \cite{smale2000construction}:
    \begin{Lemma}[\textbf{cf.} {\cite[Lemma 1]{smale2000construction}}]\label{Lemma: local minimizing for high dim} Suppose $(M,g_0)$ a smooth,closed Riemannian manifold of dimension $n+1$, with $n\geq 7$, $\Sigma\subset M$ an orientable hypersurface with $Sing(\Sigma)$ of Hausdorff dimension less or equal to $n-7$. 
    In addition, there exists $\sigma>0$ such that $\cN(\sigma)$, the tubular neighborhood of $Sing(\Sigma)$ is the finite disjoint union $\cN (\sigma)= \bigcup^k_{i=1} \cN_i(\sigma)$, and assume that there are isometrics:
    $$\Phi: \cN(\sigma) \to ( \B^{n_i+1}(\sigma)\times\Lambda_i,g_{eucl} +h_i). $$ 
    where $(\Lambda_i, h_i)$ is a compact Riemannian manifold of dimension $k_i$, and $n_i + k_i = n$, $n_i\geq 7, k_i \geq 0$.
    Furthermore, assume that
    $$  \Phi_i(\Sigma\cap \cN(\sigma)) = {\bf C_i}(\sigma)\times\Lambda_i, $$
    where $\bf C_i$ is any strictly stable, strictly minimizing, regular hypercone in $\R^{n_i+1}$. Then, there exists a metric $g$ on $M$, with $g \equiv g_0$ on $\cN (\sigma_1)$ for some $\sigma_1 <\sigma$, and $\d>0$, such that $\Sigma$ is the unique, homologically area minimizing current in $U(\d)$ relative to the metric g.
    \end{Lemma}
    
    \begin{proof}[{Proof of Theorem \ref{Theorem: main theorem 2}}]
    At first, we construct a singular homological area minimizer $\Sigma$ as described in \cite{smale2000construction}. Fix $n\geq 7$ and $p\geq 3$. We arbitrarily choose $M$ a smooth closed $(n+1)-$manifold with $S$ a smooth connected, oriented, embedded hypersurface, representing a nontrivial element of $H_n(M,\Z)$. Denote ${\bf C} := {\bf C}^{p,p}$ the Simons' cone in $\R^{2p+2}$. We first construct a Riemannian metric $g$ on $M$ and a singular hypersurface $\Sigma$ in $(M,g)$ such that $\Sigma$ is homologous to $S$.
    Denote $\cB$ an open set of $M$ such that $p\in \cB$ for some $p\in S$, and $S$ divides $\cB$ into two parts $\cB_\pm$. 

    Next, we will put a ``cap'' on the Simons' cone to make it compact. As constructed in \cite[Proposition]{smale2000construction}, note that $\p{\bf C} (1):= \p {\bf C} \cap \B^{2p+2}(1)$ is a compact embedded $2p$-manifold, and so it bounds a smooth compact $(2p+1)$-manifold $Y$. So $C(1) \cup Y$ is piecewisely smooth (away from $\{0\}$). We can approximate ${\bf C}(1) \cup Y $ to a compact hypersurface without boundary, denoted by $\hat{\bf C}\subset \R^{2p+2}$, such that $\hat{\bf C}$ is smoothly embedded in $\R^{2p+2}$ except at the origin, and a $\sigma > 0$, such that $\hat{\bf C}\cap \B^{2p+2}(\sigma) = {\bf C} \cap \B^{2p+2}(\sigma)$. Furthermore, $\hat{\bf C}$ is contained in the unit ball (by scaling if needed).

    On the other hand, $\S^{n-2p-1}$ is embedded into $\R^{n+1}$ with a trivial $(2p+2)$-normal bundle. So we have an embedding map\footnote{As remarked in \cite{smale2000construction}, we can replace sphere to any smooth, connected, compact, orientable $\Lambda$ such that there is an embedding $\overline{\B}^{2p+2}\times \Lambda \to \R^{n+1}$ with $\dim \Lambda +2p+2 =n+1$.}
    \begin{align}\label{equation: possible Lambda}
       \overline{\B}^{2p+2}\times \S^{n-2p-1}\to \R^{n+1}.
    \end{align}
    
    Theorefore, there is an embedding $\Psi: \B^{2p+2}\times \S^{n-2p-1} \to \cB_+$. Denote $\hat{\Sigma} := \Psi(\hat{\bf C}\times \S^{n-2p-1})$. Let $D$ be a $n$-disc in $\hat{\Sigma}$ which is in the image of the annulus $\A^{2p+2}(0,\frac12,1)\times \S^{n-2p-1}$, and let $D'$ be a n-disc in $S \cap \cB$. 
    Delete $D$ and $D'$ and smoothly connect $S$ and $\hat{\Sigma}$ by a handle (i.e., a hypersurface in N diffeomorphic to an $n-1$ sphere times an interval) in $\cB_+$, 
    and gluing the boundaries of the handle with $D \cup D'$. Denote $\Sigma$ be the resulting hypersurface in $M$. Note that $\Sigma $ is homologous to $S$. Finally, we will define a metric $g_0$ in $M$. For points in $\Psi(\B^{2p+2}\times \S^{n-2p-1})$, we require $g_0$ the the product metric by the pullback metric with $\Psi^{-1}$. Therefore, we can assign a metric $g_0$ on $M $ such that $g_0 = (\Psi^{-1})^*(g_{eucl} + g_S)$ on $\Psi(\B^{2p+2}(\frac12)\times \S^{n-2p-1})$. Therefore, we get a $(M,g_0)$ and $\Sigma$ that satisfy the hypotheses of Lemma \ref{Lemma: local minimizing for high dim}. So there exists a Riemannian metric $g$ on $M $ and a $\d>0$ such that $\Sigma$ is homological area minimizing in $(U_\d,g)$.

    Then we can do the same argument as in Theorem \ref{Theorem: V} to get $V$ a smooth neighborhood of $\Sigma$. Moreover, we need $V\subset U_\d$, and $\p V:=\Gamma_+ \cup \Gamma_- $ for two smooth hypersurfaces $\Gamma_+,\Gamma_-$ which are homologous to $\Sigma$. Since $\Sigma$ is orientable, for $\d$ small, $\Sigma$ splits $U_\d$ into two parts, denoted by $U_+$ and $U_-$. Consider $\cN(\sigma/8)$ the tubular neighborhood of $Sing(\Sigma)$. Now consider $(x,t) $ the Fermi coordinate on $M$ for $x\in\Sigma\setminus \cN(\sigma/8)$ and $|t|<\d$. 
    We denote $ \rho(x) :=d_{M}(x,Sing(\Sigma)) $, $\Sigma_\sigma = \{ q \in \Sigma: \rho(q) < \sigma \}$, and the constant graph on $\Sigma \setminus \Sigma_{\sigma/8}$:
    $$\Gamma_t = graph_{\Sigma \setminus \Sigma_{\sigma/8}} t,$$
    with $|t|< \d$.
    
    We need to construct smooth barriers near $Sing(\Sigma)$. Denote 
    $$\Lambda_t := \{ (x,t) : \rho(x,0) = \sigma/4 \}.$$
    $\Lambda_t$ bounds an area minimizing $n$-current $R_t$ lies in $\cN(\sigma/4)$. Consider the isometry
    $$ \Phi:(\cN(\sigma/4), g)\to (\B^{2p+2}(\sigma/4)\times \S^{n-2p-1}, g_{eucl} +g_S) .$$

    So $\Phi(R_t)$ is area minimizing in $\B^{2p+2}(\sigma/2)\times \S^{n-2p-1}$. Furthermore, by \cite[Theorem 2.1]{hardt1985area}, we have
    $$\p \Phi(R_t) = graph_{\p C(\sigma/4)} t \times \S^{n-2p-1}, \ \ \ \ \Phi(R_t) = S_t \times \S^{n-2p-1},$$  
    where we denote $S_t \subset \B^{2p+2}(\sigma/2)$ the area minimizing $n$-current with the boundary $graph_{\p C(\sigma/4)}t$.
    
    By \cite[Theorem 2.1, 5.6]{hardt1985area}, as $t\to 0$, $R_t\to \Sigma_{\sigma/4}$ in the current and Hausdorff sense, and there exists $\e_t>0$, a $C^2$ function $u_t$ on $\Sigma_{\sigma/4} \setminus \Sigma_{\e_t}$ such that $R_t \setminus \cN(\e_{t})$ can be described as a graph of $u_t$.

    On the other hand, we can find a smooth cutoff function $\chi$ supported on the annulus $\Sigma_{\sigma/4} \setminus \Sigma_{\e_t}$ such that 
    \[
    \chi(x) = 
    \begin{cases}
        1 & \text{for } \rho(x) \ge \sigma/5 , \\
        0 & \text{for } \rho(x) \le \sigma/7 .
    \end{cases}
    \]
    
    We can get a smooth hypersurface by gluing the smooth submanifold $\Gamma_t$ with $R_t$ through a function $w$ on $\Sigma \setminus\Sigma_{\sigma/8}$ by
    $$ w = \chi u_t + (1-\chi) t. $$

    Therefore, similar to Theorem \ref{Theorem: V}, for each side of $\Sigma$, we can employ any small positive (resp. negative) $t$ to find a smooth hypersurface, denoted by $\Gamma_+$ (resp. $\Gamma_-$) in $ U_+$ (resp. $U_-$), which is homologous to $\Sigma$. 
    Denote $V$ the neighborhood of $\Sigma$ bounded by $\Gamma_+, \Gamma_-$, $V$ clearly has a smooth boundary. Moreover, by the symmetry of ${\bf{C}}^{p,p}$, $\Gamma_+,\Gamma_-$ are diffeomorphic to each other. Denote $\Gamma$ the $(2p+1)$-manifold which is diffeomorphic to $\Gamma_+$ and $\Gamma_-$ by $F_\pm :\Gamma_\pm\to\Gamma $. Consider the smooth manifold defined by gluing the boundaries of $V$ and $\Gamma\times [0,R]$:
    $$M(R):=\Gamma \times [0,R]\cup V/\Gamma_{\pm}\times \{0,R\} \sim_{F_\pm} \tilde \Gamma_{\pm}.$$   
    Then using the same argument as Theorem \ref{Theorem: Get M(R)}, there exists $R_1>0$ such that for any $R>R_1$, we get Riemannian metrics $g_R$ on $M(R)$ such that $\Sigma$ is the unique homological area minimizer in $(M(R),g_R)$. Then Theorem \ref{Theorem: exists singular region}, Lemma \ref{Lem: curvature control for isoperimetric regions}, and Lemma \ref{Lemma:Perimeter,components,diameter} implies that for sufficiently large $R>0$, the boundary of the two unique isoperimetric regions (the one and its complement) with volume $|M(R)|_{g_R}/2$ is of the form
        $$\Sigma \cup [\Gamma\times \{t_0\}]. $$
        
    So we get a singular isoperimetric region in any higher dimension.
    \end{proof}

\section{Open problems}\label{Section: Open problems}

    In this section, we will discuss two open questions about isoperimetric regions. The first is about more general types of examples. The second is about the generic regularity of isoperimetric regions.

    \bigskip

    \noindent
    \textbf{1. About the existence of a singular example with non-zero mean curvature}

    There are two natural questions we may proceed with:
    
    \begin{enumerate}[wide, labelindent=0pt]
    \item\textbf{Prescribed regular tangent cones:}
        Our construction in Theorem \ref{Theorem: exists singular region} and Theorem \ref{Theorem: main theorem 2} require the tangent cones to be Simons' cones. Then we can guarantee that $\Gamma_+$ and $\Gamma_-$ are diffeomorphic to each other. If we generally choose the tangent cones as regular, strictly stable, strictly minimizing hypercones, $\Gamma_+$ may not be diffeomorphic to $\Gamma_-$. 

    \item\textbf{Prescribed mean curvature:} 
        Note that $\p \O$ is particularly composed of area-minimizing hypersurfaces. So this singular example comes from the existence of singular area minimizers. So now we come up with a new question: whether there exists a singular isoperimetric region with mean curvature constantly non-zero? In Morgan and Johnson \cite[Theorem 2.2]{morgan2000some}, we see that regardless of dimensions, if an isoperimetric region encloses a sufficiently small region, it will not have singularities (it will be a nearly round sphere.) However, for an isoperimetric with mean curvature small, a singular example is still unknown.
    \end{enumerate}

    \bigskip

    \noindent
    \textbf{2. About generic regularity of isoperimetric regions} 

    For the homological area-minimizing case, Smale in \cite{smale1993generic} shows that there exists a $C^\infty$-generic Riemannian metric such that for every nontrivial element in $H_7(M,\Z)$, there exists a (unique) smooth homological area-minimizing hypersurface.

    So we may ask whether we could generically ``smooth" the singular isoperimetric regions. Let $(M,g)$ be an $8$-dimensional closed Riemannian manifold. For $k = 3,4,...$, denote $\cM^k$ the class of $C^k$ metrics on $M$,
    endowed with the $C^k$ topology. For any $t\in (0,|M|_g) $, we define the subclass $\cF^k_t \subset \cM^k$ to be the set of
    metrics such that there is a smooth isoperimetric region with volume $t$
    (relative to the new $\overline g\in \cF^k_t$). 
    \begin{Conjecture}
        $\cF^k_t$ is dense in $\cM^k$ for any $k\ge 3$ and $ t\in (0 , |M|_g))$.
    \end{Conjecture}



\sloppy
\printbibliography
\end{document}